\documentclass[a4paper]{amsart}

\usepackage[applemac]{inputenc}
\usepackage[T1]{fontenc}
\usepackage[english]{babel}
\usepackage{babelbib}
\usepackage{enumerate}
\usepackage{url}
\usepackage{amssymb}
\usepackage[hypertexnames=false,backref=page,pdftex,
 	breaklinks=true,
 	extension=pdf,
 	colorlinks=true,
 	linkcolor=blue,
 	citecolor=red,
 	urlcolor=blue,
 ]{hyperref}
\usepackage{color}
\usepackage{tkz-fct}
\usepackage{tikz}
\usepackage{tikz-cd}
\usetikzlibrary{matrix}
\usetikzlibrary{arrows}
\usepackage{graphicx}
\usepackage{eucal}

\mathcode`A="7041 \mathcode`B="7042 \mathcode`C="7043 \mathcode`D="7044
\mathcode`E="7045 \mathcode`F="7046 \mathcode`G="7047 \mathcode`H="7048
\mathcode`I="7049 \mathcode`J="704A \mathcode`K="704B \mathcode`L="704C
\mathcode`M="704D \mathcode`N="704E \mathcode`O="704F \mathcode`P="7050
\mathcode`Q="7051 \mathcode`R="7052 \mathcode`S="7053 \mathcode`T="7054
\mathcode`U="7055 \mathcode`V="7056 \mathcode`W="7057 \mathcode`X="7058
\mathcode`Y="7059 \mathcode`Z="705A

\newcommand{\bbC}{\mathbb{C}}

\newcommand{\bbN}{\mathbb{N}}

\newcommand{\bbP}{\mathbb{P}}
\newcommand{\bbQ}{\mathbb{Q}}
\newcommand{\bbR}{\mathbb{R}}

\newcommand{\bbZ}{\mathbb{Z}}

\newcommand{\Gm}{\mathbb{G}_m}

\newcommand{\cL}{\mathcal{L}}

\newcommand{\cO}{\mathcal{O}}
\newcommand{\cP}{\mathcal{P}}

\newcommand{\cU}{\mathcal{U}}
\newcommand{\cV}{\mathcal{V}}

\newcommand{\cX}{\mathcal{X}}
\newcommand{\cY}{\mathcal{Y}}
\newcommand{\cZ}{\mathcal{Z}}

\newcommand{\rH}{\textup{H}}

\newcommand{\rR}{\textup{R}}

\newcommand{\rV}{\textup{V}}

\newcommand{\reg}{\textup{reg}}

\newcommand{\et}{\textup{\'et}}

\newcommand{\too}{\longrightarrow}
\renewcommand{\phi}{\varphi}
\renewcommand{\epsilon}{\varepsilon}
\renewcommand{\ker}{\Ker}

\renewcommand{\top}{\mathrm{top}}

\DeclareMathOperator{\Gal}{Gal}

\DeclareMathOperator{\Spec}{Spec}

\DeclareMathOperator{\im}{Im}
\DeclareMathOperator{\Ker}{Ker}

\DeclareMathOperator{\coker}{Coker}

\DeclareMathOperator{\red}{red}

\DeclareMathOperator{\GL}{GL}

\DeclareMathOperator{\Lie}{Lie}

\DeclareMathOperator{\Aut}{Aut}

\DeclareMathOperator{\growth}{growth}
\DeclareMathOperator{\grrat}{gr.rat}
\DeclareMathOperator{\grint}{gr.int}
\DeclareMathOperator{\sing}{sing}
\DeclareMathOperator{\Hilb}{Hilb}

\renewcommand{\le}{\leqslant}
\renewcommand{\ge}{\geqslant}

\theoremstyle{plain}
\newtheorem*{maintheorem}{Main theorem}
\newtheorem{theoremintro}{Theorem}
\newtheorem{theorem}{Theorem}[section]
\newtheorem{lemma}[theorem]{Lemma}
\newtheorem{proposition}[theorem]{Proposition}

\newtheorem{claim}[theorem]{Claim}

\theoremstyle{definition}
\newtheorem{definition}[theorem]{Definition}
\newtheorem{example}[theorem]{Example}

\theoremstyle{remark}
\newtheorem{remark}[theorem]{Remark}

\numberwithin{equation}{section}

\begin{document}

\author{Yohan Brunebarbe}
\address{Institut de Math\'ematiques de Bordeaux, Universit\'e de Bordeaux, 351 cours de la Lib\'eration, F-33405 Talence}
\email{yohan.brunebarbe@math.u-bordeaux.fr}

\author{Marco Maculan}

\address{Institut de Math\'ematiques de Jussieu, Sorbonne Universit\'e, 4 place Jussieu, F-75252 Paris}
\email{marco.maculan@imj-prg.fr}

\title[Counting integral points on varieties with large fundamental group]{Counting integral points of bounded height on varieties with large fundamental group}

\maketitle

\begin{abstract} The present note is devoted to an amendment to a recent paper of Ellenberg, Lawrence and Venkatesh (\cite{EllenbergLawrenceVenkatesh}). Roughly speaking, the main result here states the subpolynomial growth of the number of integral points with bounded height of a variety over a number field whose fundamental group is large. Such an improvement, \emph{i.e.} requiring large fundamental group as opposed to the existence of a geometric variation of pure Hodge structures, was already asked in \emph{op.cit.}. 

\end{abstract}

\section{Introduction}  In order to state our theorem and cast it in a general framework, let us review some background material.

\subsection{Growth of integral points} 

\subsubsection{Notation}
Let $K$ be a number field with ring of integers $\cO_K$, $X$ a projective variety over $K$, $L$ an ample line bundle on $X$, and $h$ an (absolute logarithmic) height function relative to $L$. (Further details can be found in section \ref{sec:GrowthRates}.)

For an open subset $U$ of $X$, a subring $R$ of $K$ containing $\cO_K$, a projective flat $\cO_K$-scheme $\cX$ with generic fiber $X$, and a real number $c$, set
\[ \nu(\cX, L, U, h, R; c) := \log^+ \# \{ x \in \cU(R) : h(x) \le c\}, \]
where $\cU$ is the complement of the Zariski closure of $X \smallsetminus U$ in $\cX$, and, for a real number $t$, $\log^+ t = \log \max \{ 1, t \}$. When $U= X$, the redundant specification of the open subset will be discarded.

%Obviously, when $R = K$, the additional datum of the model $\cX$ is not needed and the asymptotic behaviour of the above counting function does not depend on $h$.

\subsubsection{Asymptotic behaviour} The precise value of the function $\nu(\cX, L, U, h, K; c)$ is not much of interest. Instead, its asymptotic behaviour for $c\to \infty$ ought to reflect the usual `trichotomy' of algebraic varieties into the Fano, the Calabi-Yau, and the general type classes (and their logarithmic variants):\footnote{The word `trichotomy' here is improper---rather, according to the Minimal Model Program, any integral variety should be obtained as iterated fibration with generic fibre belonging to one of the preceding three classes.}
\begin{itemize}

\item When $X$ is smooth Fano and $L$ is the anti-canonical bundle $\omega_X^\vee$, Manin  conjectured (assuming the Zariski density of $X(K)$) the existence of an open subset $U$ of $X$ for which 
\[ \nu(X, \omega_X^\vee, U, K; c) \sim [K : \bbQ] c + (\rho - 1) \log c + C,\]
where $\rho$ is the Picard rank of $X$ and $C$ is a real number.\footnote{As of nowadays, Manin's conjecture is known not to hold in such a form (\cite{BatyrevTschinkel}) and alternative hypothetical statements have been suggested. This does not affect  the main theme of our paper, and the interested reader may consult for example \cite{Peyre}.} As an instance of the conjecture, Schanuel (\cite{Schanuel}) earlier proved the existence of an explicit real number $C = C(N, K)$ for which
\[\nu(\bbP^N_K, \cO(1),  K; c) \sim [K : \bbQ] (N+1) c + C. \]

\item Rather vaguely,\footnote{To the extent of our knowledge there is no precise conjecture on the behaviour of the growth of rational points for Calabi-Yau varieties.} logarithmic growth is expected for Calabi-Yau varieties. For example, let $A$ be an abelian variety and
$ r := \dim_{\bbQ} A(K) \otimes_{\bbZ} \bbQ$
its Mordell-Weil rank. Then, for any ample line bundle $L$ on $A$, there is a positive real number $C = C(A,K,L) > 0$ such that
\[ \nu(A, L, K; c) \sim \tfrac{r}{2} \log c  + C;\]
see \cite[Theorem B.6.3]{HindrySilverman}.

\item As soon as one enters the general type realm, the Lang-Vojta conjecture in the projective case predicts the existence of a non-empty open subset for which $K'$-rational points are only finitely many for any finite extension $K'$ of $K$---in other words, the counting function of rational points for such an open subset is eventually constant. This is the case for curves by a celebrated theorem of Faltings (\cite{FaltingsMordell}, \cite{FaltingsMordellErratum}).

Reformulating the Lang-Vojta conjecture in the non-compact case in terms of counting functions is less eloquent. Nonetheless, recall the finiteness of $S$-integral points for affine curves---a theorem of Siegel (\cite{Siegel})---and for (fine) moduli spaces of curves and of abelian varieties---the Shafarevi\v{c} conjecture, proved by Faltings in \emph{op.cit.}.
\end{itemize}

Above, with an abuse of notation, the chosen model $\cX$ and height $h$ have been dropped, for the asymptotic behaviour is independent from them. Also, the factor $[K:\bbQ]$ appears because of the normalization of height adopted in this paper. 

In another direction, Pila (\cite{Pila}) proves that, given integers $N, d \ge 1$ and a positive real number $\epsilon > 0$, there is a real number $C= C(N, d, \epsilon)$ such that, for each closed integral subvariety $X$ of $\bbP^N_{\bbQ}$ of degree $d$,
\[ \nu(X, \cO(1), \bbQ ; c)\le  (\dim X + \tfrac{1}{d}+ \epsilon)c + C. \]
See also Browning, Heath-Brown and Salberger (\cite{BrowningHeathBrownSalberger}), and Castryck, Cluckers, Dittmann and Nguyen (\cite{CastryckCluckersDittmanNguyen}).

\subsubsection{Linear growth} Here we will content ourselves with much cruder estimates---namely, the slope of the function $c \mapsto \nu(\cX, L, U, h, K; c)$,
\[ \grrat_K(X,L, U) := \limsup_{c\to \infty} \frac{\nu(\cX, L, U, h, K; c)}{c},\]
henceforth called the \emph{(linear) growth rate of rational points}.  Of course, as the notation suggests, the real number $\grrat_K(X,L, U)$ does not depend on the choice of the model $\cX$ nor on that of the height function $h$. 

Similarly, for a finite set $S$ of places of $K$ including the Archimedean ones, the real number
\[ \limsup_{c\to \infty} \frac{\nu(\cX, L, U, h, \cO_{K,S}; c)}{c} \]
measures the presence of the $S$-integral points on $U$. Again this does not depend on $\cX$ and $h$, but \emph{a priori} does depend on the set $S$. Taking the supremum ranging over all such finite sets of places $S$ allows to get rid of such a dependence and gives rise to an invariant
\[ \grint_K(X, L, U)\]
called the \emph{(linear) growth rate of integral points} of $U$ with respect to $X$ and $L$. 

When $U = X$, rational and integral points coincide, thus their growth rate do. As above, in such a case the redundant repetition of $X$ is discarded from notation.  

These growth rates furnish invariants way rougher than the asymptotic behaviour. To stress this, note that the following holds:

\begin{itemize}
\item $\grrat_K(\bbP^N_K, \cO(1)) = (N + 1)[K : \bbQ]$ for an integer $N \ge 1$;
\item $\grrat_K(A, L)= 0$ for an abelian variety $A$ over $K$ and an ample line bundle $L$ on $A$;
\item $\grint_K(X, L, U) = 0$ for a smooth projective curve $X$ over $K$, an ample line bundle $L$ on $X$ and an open subset $U$ of $X$, provided that $U$ is not isomorphic to the projective or the affine line over $K$.
\end{itemize}

In particular, linear growth rates cannot distinguish Calabi-Yau varieties from those of general type. Needless to say, zero linear growth rate is a much weaker condition than logarithmic growth or no growth at all.

\subsection{Large fundamental groups} Abandon the notation above and let $X$ be a normal integral variety over an algebraically closed field  $k$ of characteristic $0$. 

\subsubsection{Definition}

Adopting Koll\'ar's terminology (\cite{KollarPaper}, \cite{KollarBook}), the \'etale fundamental group of $X$ is \emph{large} if, for any closed positive-dimensional integral subvariety $Y$ of $X$ with normalization $f \colon \tilde{Y} \to Y$, the image of the induced map\footnote{We abusively drop the choice of a base-point for (\'etale) fundamental groups.} $\pi_1^\et(\tilde{Y}) \to \pi_1^\et(X)$ is infinite. (See also \cite{Campana}.)

\begin{remark}
Thanks to \cite[Proposition 2.9.1]{KollarPaper}, the \'etale fundamental group of $X$ is large if and only if, for any non-constant morphism $f \colon Y \to X$ with $Y$ integral normal, the image of the induced map $\pi_1^\et(Y) \to \pi_1^\et(X)$ is infinite. Furthermore, it is sufficient (thus equivalent) to test the latter condition only on smooth connected  curves $Y$.
\end{remark}

\subsubsection{Basic properties} The following are direct consequences of the definition and the above remark:
\begin{enumerate}
\item The class of (integral, normal) algebraic varieties with large \'etale fundamental group is closed under products. 

\item Given a quasi-finite morphism $f \colon X' \to X$ between integral normal algebraic varieties, if $X$ has a large \'etale fundamental group, then $X'$ has a large \'etale fundamental group too. The converse holds as soon as $f$ is finite \'etale.

\item The \'etale fundamental group of an isotrivial fibration whose base and fiber have large \'etale fundamental group is large.\footnote{Namely, let $f \colon X' \to X$ be a surjective morphism between integral normal varieties whose geometric fibers are all isomorphic to some fixed integral normal variety $F$.  If $X$ and $F$ have large \'etale fundamental group, then $X'$ has.}

\item Let $k'$ be an algebraically closed field extension of $k$. If the \'etale fundamental group of the variety $X'$ deduced from $X$ by extending scalars to $k'$ is large, then so is the one of $X$.\footnote{Applying the specialization map of \'etale fundamental groups to an exhaustive family of normal cycles (see \ref{sec:FamiliesNormalCycles}) permits to show that the converse holds when $X$ is proper. In the non-compact case we ignore whether having large \'etale fundamental group is a property compatible with extension of scalars.}

\item All smooth connected curves, except the affine and the projective line, have large \'etale fundamental group. 
\end{enumerate}

\begin{remark} Taking products of a suitable number of copies of an elliptic and of a curve of genus $\ge 2$ yields examples of smooth projective (connected) varieties $X$ with large \'etale fundamental group and any possible Kodaira dimension between $0$ and $\dim X$.
\end{remark}

\subsubsection{Comparison with the topological fundamental group} Over the complex numbers, the property of having large \'etale fundamental group  can be read off the usual topological fundamental group. 

For, upon fixing a point $x \in X(\bbC)$, the natural group homomorphism  \[i_X \colon \pi_1^\top(X(\bbC), x) \too \pi_1^\et(X, x)\] identifies the \'etale fundamental group with the profinite completion of the topological one. Let $\hat{X}$ be the topological cover of $X(\bbC)$ corresponding to the normal subgroup $\ker i_X$ of $\pi_1^\top(X(\bbC), x)$. 

\begin{proposition}[{\cite[Proposition 2.12.3]{KollarPaper}}] \label{Prop:TopologicalCharacterizationLargeFundamentalGroup}Suppose $X$ proper.\footnote{The properness assumption is missing in the statement of \emph{loc.cit.}, although crucially invoked in the proof---in the non-compact case the statement is false; see Remark \ref{Rmk:CounterexampleKollar}.} Then, the \'etale fundamental group of $X$ is large if and only if the complex analytic space $\hat{X}$ does not contain positive-dimensional compact complex analytic subspaces.
\end{proposition}

The latter condition is satisfied, for instance, if the complex space $\hat{X}$ is holomorphically separable;\footnote{That is, given two distinct points $x, y$, there is a global holomorphic function $f$ such that $f(x) \neq f(y)$.} examples of holomorphically separable spaces are open subsets of $\bbC^n$ (or, more generally, open subsets of Stein spaces).

\begin{remark}

In view of the above characterization, it is useful to be able to determine $\hat{X}$. One idle (but useful) case is when $\hat{X}$ is itself a universal cover of $X(\bbC)$, which boils to down to saying that $i_X$ is injective.\footnote{An instance where this does not occur has been first observed by Toledo (\cite{Toledo}).} Recall that a group $\Gamma$ is said to be:
\begin{itemize}
\item \emph{linear} if it admits a faithful representation $\rho \colon \Gamma \to \GL(V)$, where $V$ is a finite-dimensional vector space over some field;
\item \emph{residually finite} if the natural map $\Gamma \to \hat{\Gamma}$, where $\hat{\Gamma}$ is its profinite completion, is injective.
\end{itemize}

A classical result of Malcev (\cite{Malcev}) states that a finitely generated linear group is residually finite. Gathering the above considerations, if the topological fundamental group $\pi_1^\top(X(\bbC), x)$ is linear, then $\hat{X}$ is a universal cover of $X(\bbC)$.
\end{remark}

\begin{example} \label{Ex:SemiAbelianBoundedSymm} The above remark can be applied to say that the \'etale fundamental group of $X$ is large when $X$ is
\begin{itemize}
\item an abelian variety, for it is the quotient of $\bbC^n$ by a lattice;
\item a quotient of a bounded symmetric domain in $\bbC^n$ by a torsion-free co-compact lattice of its biholomorphism group.
\end{itemize}

It follows that semi-abelian varieties, being a fibration over an abelian variety by a torus, have large \'etale fundamental group.
\end{example} 

\begin{remark} \label{Rmk:CounterexampleKollar} Having large \'etale fundamental group is \emph{not} a biholomorphic invariant of non-compact varieties, preventing Proposition \ref{Prop:TopologicalCharacterizationLargeFundamentalGroup} to hold true in such a case. 

For instance, let $G$ be the universal vector extension of a complex elliptic curve $E$. The exponential map $\Lie G \to G(\bbC)$ is a universal cover and its kernel $\Lambda$ is a rank $2$ free abelian group. The topological fundamental group of $G(\bbC)$ is naturally identified with $\Lambda$, and therefore residually finite. Despite $\hat{G} = \Lie G$ being a Stein space, the \'etale fundamental group of $G$ is \emph{not} large: the kernel of the natural projection $G \to E$ is the affine line, which is simply connected.

On the other hand, the discrete subgroup $\Lambda$ generates $\Lie G$ as a complex vector space, so that $G (\bbC) = \Lie G / \Lambda$ is biholomorphic to $(\bbC/ \bbZ)^2 \cong (\bbC^\times)^2$. The algebraic group $\Gm^{2}$ has large fundamental group, whence the sought-after example.
\end{remark}

\subsubsection{The role of local systems} Still under the assumption $k = \bbC$, a local system $\cL$ on $X(\bbC)$ with coefficients in some field is \emph{large} if, given a non-constant morphism $f \colon Y \to X$ with $Y$ a normal irreducible complex variety, the local system $f^\ast \cL$ has infinite monodromy. The \'etale fundamental group of a complex variety carrying a large local system is large. 

Local systems underlying variations of Hodge structures are the prominent example of large local systems.  More precisely, let $\cL$ be a real local system on $X(\bbC)$ underlying a polarizable variation of pure real Hodge structures (or, more generally, an admissible graded-polarizable variation of mixed real Hodge structures, \emph{cf.} \cite{SteenbrinkZucker}). Assume further that the associated period mapping has discrete fibers. Then, the theorem of the fixed part (see \cite{GriffithsPeriodsIII}, \cite{DeligneHodgeII}, \cite{Schmid} in the pure case, and \cite{SteenbrinkZucker} in the mixed case)  implies that $\cL$ is large.

\begin{example} \label{Ex:QuotientBoundedSymm} Mixed Shimura varieties carry a large local system coming from their interpretation as period spaces for   (graded polarized) integral mixed Hodge structures.
\end{example}

\begin{example} \label{Ex:ModuliSpaces} A full range of (fine) moduli spaces of polarized varieties (\emph{e.g.} smooth projective curves, Calabi-Yau varieties, most complete intersections) admit a large local system---namely, the one whose fiber at a point is the middle cohomology of the corresponding variety. Indeed, such a local a system is large when the `infinitesimal Torelli theorem' is satisfied, whence the above list;  see \cite{BeauvilleTorelli}.
\end{example}

\subsection{Main result} Time has come to state the main result of the present note:

\begin{maintheorem} \label{ThmIntro:MainThm} Let $K$ be a number field,  $X$ a projective variety over $K$, $L$ an ample line bundle over $X$, and $U$ an open subset of $X$ which is geometrically integral, normal, and whose geometric \'etale fundamental group is large.\footnote{That is, for an algebraic closure $\bar{K}$ of $K$, the variety $\bar{U}$ deduced from $U$ extending the scalars to $\bar{K}$ has large \'etale fundamental group.} Then,
\[ \grint_K(X, L, U) = 0.\]
\end{maintheorem}

Such a statement can be applied to varieties appearing in examples \ref{Ex:SemiAbelianBoundedSymm}, \ref{Ex:QuotientBoundedSymm} and \ref{Ex:ModuliSpaces}; in particular, this replies positively to a question raised in \cite[p.3-4, end of \S 1.1]{EllenbergLawrenceVenkatesh}.

Its proof relies on two independent ingredients. The first is of arithmetic nature and is a quite formal consequence of work on uniform bounds initiated by the determinant method of Heath-Brown (\cite{HeathBrown}), and pursued by many authors including Broberg (\cite{Broberg}), Salberger (\cite{Salberger}), and Chen (\cite{ChenCrelleI}, \cite{ChenCrelleII}):

\begin{theoremintro} \label{ThmIntro:RationalGrowthVarietiesSmallDegree} Let $X$ be an integral projective variety of dimension $n$ over a number field $K$, and $L$ an ample line bundle over $X$. Assume there is a closed subvariety $Z$ of $X$ and an integer $d \ge 1$ such that any positive-dimensional integral closed subvariety $Y$ of $X$ not contained in $Z$ satisfies $\deg(Y, L_{\rvert Y}) \ge d^{\dim Y}$. Then,
\[ \grrat_K(X, L, X \smallsetminus Z) \le \frac{n(n + 3)}{2 d} [K:\bbQ].\]
\end{theoremintro}

On the other hand, the second one is purely geometric and already appears in arguments of the first-named author (\cite{Brunebarbe}):

\begin{theoremintro} \label{ThmIntro:GeometricPart} Let $X$ be a normal integral projective variety over an algebraically closed field of characteristic $0$, $L$ an ample line bundle over $X$, and $U$ a non-empty open subset of $X$ whose \'etale fundamental group is large.

Then, given an integer $d \ge 1$, there is a finite surjective map $\pi \colon X' \to X$ with $X'$ normal integral such that $\pi$ is \'etale over $U$ and, for each positive-dimensional integral closed subvariety $Y'$ of $X'$ meeting $\pi^{-1}(U)$, 
\[ \deg(Y', \pi^\ast L_{\rvert Y'}) \ge d.\]
\end{theoremintro}

(See \ref{sec:Degree} for a reminder on the degree of a line bundle.)

\subsubsection*{Acknowledgements} We warmly thank Pascal Autissier and Ariyan Javanpeykar for their useful comments on the first version of this paper.

\subsubsection*{Organization of the paper} Introduction left aside, the paper has three sections, respectively dealing with the proof of Theorem \ref{ThmIntro:RationalGrowthVarietiesSmallDegree}, Theorem \ref{ThmIntro:GeometricPart} and the Main theorem.

\subsubsection*{Conventions} A \emph{variety} over a field $k$ is a separated finite type $k$-scheme. %A normal variety is always understood to be integral.

\section{Geometry}

\subsection{Degree of the singular locus} 

\subsubsection{Degree} \label{sec:Degree} For a proper variety $X$ over a field $k$ together with an ample line bundle $L$, let
\[ \deg(X, L) := (\dim X)! \lim_{i \to \infty} \frac{\dim_k \Gamma(X, L^{\otimes i})}{i^{\dim X}}.\]

The theory of Hilbert polynomials shows that such a limit exists and is a positive rational number. When $X$ is integral, the asymptotic version of Riemann-Roch states
\[ \deg(X, L) = L^n,\]
where $n= \dim X$ and $L^n$ is the top self-intersection of $L$. In particular, $\deg(X, L)$ is a positive integer.

\begin{lemma} \label{Lemma:DegreeAndBirationalMorphisms} Let $\pi \colon Y \to X$ be a finite morphism between proper varieties over $k$ and $L$ an ample line bundle over $X$. If there exists a scheme-theoretically dense open subset $U$ of $X$ such that $\pi^{-1}(U) \to U$ is an isomorphism, then
\[ \deg(Y, \pi^\ast L) = \deg(X, L).\]
\end{lemma}

\begin{proof} Since the open subset $U$ is scheme-theoretically dense, the natural map $\phi \colon \cO_X \to \pi_\ast \cO_Y$ is injective. The support of the coherent $\cO_X$-module $F = \coker \phi$ is contained in the closed subset $X \smallsetminus U$, thus in particular of dimension $< \dim X$. For $i \ge 1$ big enough, the cohomology group $\rH^1(X, L^{\otimes i})$ vanishes, yielding a short exact sequence
\[ 0 \too \Gamma(X, L^{\otimes i}) \too \Gamma(Y, \pi^\ast L^{\otimes i}) \too \Gamma(X, F \otimes  L^{\otimes i}) \too 0.\]
Since the support of $F$ has dimension $< \dim X$,
\[ \lim_{i \to \infty} \frac{\dim \Gamma(X, F \otimes  L^{\otimes i})}{i^{\dim X}} = 0.\]
The dimensions of $X$ and $Y$ being the same, the result follows.
\end{proof}

\begin{lemma} \label{Lemma:DegreeIrreducibleComponents} Let $X$ be an equidimensional proper variety over $k$, $L$ an ample line bundle. Then,
\[ \deg (X, L) \ge \sum_{X'} \deg(X', L_{X'}),\]
the sum ranging on the irreducible components of $X$ endowed with the reduced structure. Moreover, equality holds if $X$ is reduced.
\end{lemma}

\begin{proof} One reduces first to $X$ reduced, as for $i \ge 0$ big enough, the restriction map $\Gamma(X, L^{\otimes i}) \to \Gamma(X_{\red}, L^{\otimes i})$ is surjective. Under the assumption of $X$ being reduced, let $X_1, \dots, X_n$ be the irreducible components of $X$. Lemma \ref{Lemma:DegreeAndBirationalMorphisms}, applied to the finite morphism $X_1 \sqcup \cdots \sqcup X_n \to X$, gives the desired result.
\end{proof}

\subsubsection{Degree in families} Let $f \colon X \to S$ be a proper morphism between algebraic varieties over a field $k$ and $L$ a relatively ample line bundle on $X$. Consider:
\begin{itemize}
\item the function $\delta_{X/S, L} \colon S \to \bbN$,
\[ s \longmapsto \max \{ \deg (Z, L_{\rvert Z}) : Z \textup{ irreducible component of } X_s \},\]
where $Z$ is endowed with its reduced structure, and 
\item the function $\iota_{X/S} \colon S \to \bbN$ associating to $s \in S$ the number of irreducible components of $X_{\bar{s}}$ where $\bar{s}$ is a geometric point over $s$.
\end{itemize}

\begin{lemma} \label{Lemma:DegreeInFamily} Let $f \colon X \to S$ be a proper morphism between varieties over $k$ and $L$ relatively ample line bundle on $X$. 

Then, the functions $\delta_{X/S, L}, \iota_{X/S} \colon S \to \bbN$ are bounded above.
\end{lemma}

\begin{proof} The statement only depends on the reduced structure of $S$. Moreover, by treating separately each irreducible component, the scheme $S$ may be supposed to be integral. By Noetherian induction, it suffices to show the statement on a nonempty open subset of $S$. 

Let $\eta$ be the generic point of $S$ and $X_{\eta, 1}, \dots, X_{\eta, r}$ the irreducible components of $X_\eta$. For $i = 1, \dots, r$, let $X_i$ be the closure of $X_{\eta, i}$ in $X$. According to \cite[\href{https://stacks.math.columbia.edu/tag/054Y}{Lemma 054Y}]{stacks-project}, there is an open subset $S'$ of $S$ such that 
\[ f^{-1}(S') \subset X_1 \cup \cdots \cup X_r.\]

Up to treating each of the $X_i$ separately and up to replacing $S$ by $S'$, one may assume $X$ to be integral. Upon shrinking $S$, the morphism $f$ may be assumed to be flat (\cite[\href{https://stacks.math.columbia.edu/tag/052B}{Proposition 052B}]{stacks-project}). By flatness (\cite[\href{https://stacks.math.columbia.edu/tag/02JS}{Lemma 02JS}]{stacks-project}), the fibers of $f$ are then pure of dimension $d := \dim X - \dim S$.

The relative ampleness of $L$ implies the existence of an integer $i_0 \ge 1$ such that, for $q \ge 1$ and $i \ge i_0$, the higher direct image $\rR^q f_\ast L^{\otimes i}$ vanishes. The semi-continuity theorem (\cite[Theorem III.12.8]{HartshorneAlgebraicGeometry}) can be applied to say, for $i, q \in \bbN$, that the function $s \mapsto \dim_{\kappa(s)} \rH^q(X_s, L^{\otimes i})$ is upper semi-continuous. Thus, up to replacing $L$ by $L^{\otimes i_0}$ and $S$ by a nonempty open subset, one may assume, for $i \ge 1$, $q \ge 1$ and $s \in S$,
\[  \rH^q(X_s, L^{\otimes i})=0.\]
On the other hand, the Hilbert polynomial of $L_{\rvert X_s}$ does not depend on $s$, for the morphism $f$ is flat (\cite[Theorem 5.10]{FGAExplained}). Combined with the vanishing of higher cohomology, this implies that, for $i \in \bbN$, the function $s \mapsto \dim_{\kappa(s)} \Gamma(X_s, L^{\otimes i})$ is constant. As a consequence, so is the function
\[ s \longmapsto \deg(X_s, L_{\rvert X_s}) = d! \lim_{i \to\infty} \frac{\dim_{\kappa(s)} \Gamma(X_s, L^{\otimes i})}{i^d}. \]
For $s \in S$, the inequality
\[ \deg(X_s, L_{\rvert X_s}) \ge \sum_{Z} \deg(Z, L_{\rvert Z}), \]
where the sum ranges on the irreducible components of $X_s$ endowed with the reduced structure, shows that the function $\delta_{X/S, L}$ is bounded above.

On the other hand, the degree is invariant under extension of scalars, thus
\[ \deg(X_s, L_{\rvert X_s}) \ge \sum_{\bar{Z}} \deg(\bar{Z}, L_{\rvert \bar{Z}}), \]
where the sum ranges on the irreducible components of $X_{\bar{s}}$ endowed with the reduced structure, where $\bar{s}$ is a geometric point over $s$. Since $ \deg(\bar{Z}, L_{\rvert \bar{Z}})$ is a positive integer, the right-hand side of the above inequality is bounded below by $\iota_{X/S}(s)$. This shows that the function $\iota_{X/S}$ is bounded above.
\end{proof}

\begin{proposition} \label{Prop:DegreeSingularLocus} For integers $N, D \ge 1$, there is an integer $R_D = R_D(N)$ such that, for a field $k$, a closed subvariety $X$ of $\bbP^N_k$ of degree $\le D$, the following statements hold:
\begin{enumerate}
\item the number of irreducible components of $X$ is $\le R_D$;
\item any irreducible component $Z$ of its singular locus $X^{\sing}$ (endowed with the reduced structure) has degree $\deg(Z, \cO_{\bbP^N}(1)) \le R_D$;
\item for an algebraic closure $\bar{k}$ of $k$, the number of irreducible components of $X^{\sing}_{\bar{k}}$ is $\le R_D$.
\end{enumerate}
\end{proposition}

\begin{proof} According to \cite[Exp. XIII, Corollaire 6.11 (ii)]{SGA6}, it suffices to prove the statement when the subvarieties in question have a fixed Hilbert polynomial $P \in \bbQ[z]$. For, consider the Hilbert scheme \[S = \Hilb_{\bbP^N_\bbZ, \cO(1)}^P\] of the closed subschemes of $\bbP^N_\bbZ$ with Hilbert polynomial $P$ with respect to $\cO(1)$. Let $X \subseteq \bbP^N_S$ be the universal family. The morphism $\pi \colon X \to S$ induced by the second projection is proper and flat. Let $n$ be its relative dimension.

For (1), apply Lemma \ref{Lemma:DegreeInFamily} to the morphism $\pi \colon X \to S$ and the relatively ample line bundle $L:= \cO_{\bbP^N_S}(1)_{\rvert X}$. 

For (2) and (3), let $X^{\sing}$ be the closed subset in $X$ where the coherent $\cO_X$-module $\Omega_{X/S}$ has rank $\ge n + 1$. One concludes by applying Lemma \ref{Lemma:DegreeInFamily} to the morphism $\pi \colon X^{\sing} \to S$ and the relatively ample line bundle $L:= \cO_{\bbP^N_S}(1)_{\rvert X^{\sing}}$.
\end{proof}

\subsection{Families of normal cycles} \label{sec:FamiliesNormalCycles}Let $k$ be a field of characteristic $0$ and $X$ a geometrically integral normal variety over $k$. 

\subsubsection{Definition} Following Koll\'ar (\cite{KollarPaper}), a \emph{normal cycle} on $X$ is a finite morphism $f \colon Z \to X$  which is birational onto its image, where $Z$ is a geometrically integral normal variety. 

A \emph{family of normal cycles} on $X$ is the datum of morphism of $k$-schemes $\pi \colon \cZ \to S$ and $f \colon \cZ \to X$ where 
\begin{itemize}
\item $S$ is reduced and a countable disjoint union of varieties over $k$;
\item the morphism $\pi$ is separated, of finite type, flat and with geometrically integral normal  fibers;
\item for $s \in S$, the map $f_s \colon \cZ_s \to X_{\kappa(s)}$ is a normal cycle, where $\kappa(s)$ is the residue field at $s$.
\end{itemize}

\subsubsection{Exhaustive families} A family of normal cycles $(\pi \colon \cZ \to S, f \colon \cZ \to X)$ on $X$ is \emph{exhaustive}\footnote{Koll\'ar uses the name `weakly complete'.} if, given a field extension $k'$ of $k$ and a normal cycle $g \colon Z \to X_{k'}$, there is a unique $s \in S(k')$ such that $g = f_s$.

Given an exhaustive family of normal cycles $(\pi \colon \cZ \to S, f \colon \cZ \to X)$ on $X$ and an open immersion $j \colon U \to X$, the couple \[(\pi \colon \cZ \times_X U \to \pi(\cZ \times_X U), f \colon \cZ \times_X U \to U)\] is an exhaustive family of normal cycles on $U$.

\begin{proposition} \label{Prop:ExistenceExhaustiveFamily} If $X$ is projective, then there exists an exhaustive family of normal cycles $(\pi \colon \cZ \to S, f \colon \cZ \to X)$ on $X$ such that, for any ample line bundle $L$ on $X$ and any integer $d \ge 1$, the set 
\[ S_{L, d} := \{ s \in S : \deg(\cZ_s, f_s^\ast L) \le d \}\]
is a finite union of connected components of $S$ (thus of finite type).
\end{proposition}

For the proof, which reproduces that of \cite[Proposition 2.4]{KollarPaper}, the following lemmas will be needed:

\begin{lemma} \label{Lemma:GenericNormalization}Let $\pi \colon Y \to S$ be a morphism between varieties over $k$ with  integral generic fibers and $S$ reduced. Let $\nu \colon \tilde{Y} \to Y$ the normalization of $Y$. 

Then, there is a dense open subset $S'$ of $S$ such that, for $s \in S'$, the induced morphism $\tilde{Y}_s \to Y_s$ is the normalization.
\end{lemma}

\begin{proof} First, one may assume $Y$ to be reduced; for, the ideal sheaf of nilpotent elements of $Y$ is supported on a closed subset meeting no generic fiber of $\pi$.

Let $S_0$ be the open subset of $S$ such that the fibers of $\tilde{\pi} \colon \tilde{Y} \to S$ are (geometrically) normal (\cite[Corollaire 9.9.5]{EGA4III}). Note that $S_0$ contains each of the generic points of $S$ because the generic fibers of $\tilde{\pi}$ are normal.

The morphism $\nu$ is finite birational, the latter meaning that $\nu$ is an open immersion on a dense open subset $U$ of $\tilde{Y}$. For each $s \in S$, the induced map $\nu_s \colon \tilde{Y}_s \to Y_s$ is finite, whereas this will not be the case in the general for the `birational' property. To remedy that, remark that the image $\tilde{\pi}(U)$ is constructible and contains all of the generic points of $S$. Therefore, the subset $S_0 \cap \tilde{\pi}(U)$ contains a dense open subset $S'$.
\end{proof}

\begin{lemma} \label{Lemma:BreakingIntoPieces} Let $f \colon Y \to S$ be a morphism between varieties over $k$ with geometrically integral fibers and $S$ reduced. 

Then, there is a surjective immersion of varieties $\iota \colon S' \to S$ and a finite morphism $\nu \colon Y' \to Y \times_S S'$ such that the composite map $Y' \to S'$ is flat and, for each $s' \in S'$, the induced map $Y'_{s'} \to Y_{\iota(s')}$ is the normalization.
\end{lemma}

Note that the morphism $\iota$ being a surjective immersion simply means that $S'$ is the disjoint union of finitely many pairwise disjoint locally closed subsets of $S$ whose union is the whole $S$.

\begin{proof} Arguing by Noetherian induction, it suffices to construct a dense open subset $S_0$ of $S$ and a finite morphism $Y_0 \to Y \times_S S_0$ with the properties in the statement. This is obtained combining Lemma \ref{Lemma:GenericNormalization} with the openness of the locus where a morphism is flat.
\end{proof}

\begin{proof}[{Proof of Proposition \ref{Prop:ExistenceExhaustiveFamily}}] Let $H$ be a connected component of the Hilbert scheme $\Hilb(X)$ of $X$ endowed with the reduced structure. Let $u \colon \cU \to H$ be the base-change to $H$ of  the universal family of closed subschemes of $X$ parameterized by $\Hilb(X)$. 

The locus where the fibers of the map $u$ are geometrically integral is constructible (\cite[Th\'eor\`eme 9.7.7]{EGA4III}); let $H'$ be disjoint union of the irreducible components of such a constructible subset, again endowed with the reduced structure. The base-change $u' \colon \cU' \to H'$ of $u \colon \cU \to H$ along the immersion $H' \to H$ satisfies the hypotheses of Lemma \ref{Lemma:BreakingIntoPieces}. Applying it yields a surjective immersion $\iota \colon S_H \to H'$ and finite morphism $\nu \colon \cZ_H \to \cU' \times_{H'} S_H$ such that the composite map $\pi_H \colon \cZ_H \to S_H$ is flat with normal geometrically integral fibers and, for each $s \in S_H$, the induced morphism $\nu_{s} \colon \cZ_{H, s} \to \colon \cU'_{\iota(s)}$ is the normalization.

The sought-for exhaustive family of normal cycles on $X$ is
\[ 
\pi = \bigsqcup_{H} \pi_H \colon \cZ := \bigsqcup_{H} \cZ_H \too S := \bigsqcup_{H} S_H,
\]
the disjoint unions running over all the connected components of the Hilbert scheme of $X$ endowed with the reduced structure. 

For an ample line bundle $L$ and an integer $d \ge 1$ there are only finitely many connected components of the Hilbert scheme parameterizing varieties of degree $\le d$, whence the result.
\end{proof}

\begin{proposition} \label{Prop:MakingAFamilyOfCycleTopologicallyTrivial} Suppose $k = \bbC$ and $X$ projective. Let $(\pi \colon \cZ \to S, f \colon \cZ \to X)$ be an exhaustive family of normal cycles on $X$. 

Then, there is a surjective immersion of finite type $\iota \colon S' \to S$  of $\bbC$-schemes such that, for any connected component $T$ of $S$, the induced map $(\cZ' \times_{S'} T') (\bbC) \to T'(\bbC)$ is a topological fiber bundle, where $T' = \iota^{-1}(T)$ and $\cZ' = \cZ \times_S S'$.
\end{proposition}

\begin{proof} See the proof of \cite[Proposition 2.4]{KollarPaper}.
\end{proof}

\begin{proposition} \label{Prop:FinitenessConjugationClasses}Suppose $k$ is an algebraically closed subfield of $\bbC$. Let $L$ be an ample line bundle over $X$, $U$ a nonempty open subset of $X$ and $d \ge 1$ an integer. 

Then, up to conjugation, there are only finitely many subgroups of $\pi_1^\et(U)$ obtained as the image of $\pi^\et_1(f^{-1}(U)) \to \pi_1^\et(U)$ with $f \colon Z \to X$ a normal cycle such that 
\[  \deg(Z, f^\ast L) \le d. \]
\end{proposition}

\begin{proof} Since the geometric \'etale fundamental group is insensible to extension of scalars, and since there are more normal cycles on $X_{\bbC}$ than on $X$, there is no loss of generality in assuming $k = \bbC$. 

Let $(\pi \colon \cZ \to S, f \colon \cZ \to X)$ be an exhaustive family of normal cycles satisfying the property in the statement of Proposition \ref{Prop:ExistenceExhaustiveFamily}. According to Proposition \ref{Prop:MakingAFamilyOfCycleTopologicallyTrivial}, up to taking a surjective immersion of finite type, one may suppose that, for each connected component $T$ of $S$, the induced map $\pi^{-1}(T)(\bbC) \to T(\bbC)$ is a topological fiber bundle. 

Let $S_U$ be the image of $\cZ_U := f^{-1}(U)$ in $S$ via $\pi$. Then \[(\pi \colon \cZ_U \to S_U, f \colon \cZ_U \to U)\] is an exhaustive family of normal cycles on $U$; moreover, for each connected component $T$ of $S_U$, the induced map $\pi^{-1}(T)(\bbC) \to T(\bbC)$ is again a topological fiber bundle. It follows that, for  $t, t' \in T(\bbC)$, the images of $\pi_1^\top(\cZ_{U, t}(\bbC))$ and $\pi_1^\top(\cZ_{U, t}(\bbC))$ in $\pi_1^\top(U(\bbC))$ are conjugated subgroups. Passing to profinite completions, the corresponding affirmation holds for \'etale fundamental groups.

Let $C_{L, d}$ be the set of conjugacy classes  of subgroups of $\pi_1^\et(U)$ obtained as the image of $\pi^\et_1(g^{-1}(U)) \to \pi_1^\et(U)$ with $g \colon Z \to X$ a normal cycle such that 
\[  \deg(Z, g^\ast L) \le d. \]
With the notation of Proposition \ref{Prop:ExistenceExhaustiveFamily}, the argument above shows the inequality
\[ \# C_{L, d} \le \# (\textup{connectected components of }S_{L, d})\]
and the right-hand side is $< \infty$ by the cited result.
\end{proof}

\subsection{Proof of Theorem \ref{ThmIntro:GeometricPart}} Let us begin with two easy facts:

\begin{lemma} \label{Lemma:DegreeOfGaloisCoversOnSubvarieties} Let $\pi \colon X' \to X$ be a finite surjective morphism between integral normal varieties over $k$ which is Galois of group $G$ and \'etale over a non-empty open subset $U$ of $X$. Let $Y'$ an integral closed subvariety of $X'$ and $Y:= \pi(Y')$. Then, the map $\pi_{\rvert Y'} \colon Y' \to Y$ has degree
\[ \# \im (\pi_1^\et(\nu^{-1}(U \cap Y)) \to G),\]
where $\nu \colon \tilde{Y} \to Y$ is the normalization.
\end{lemma}

\begin{proof} Since the degree is computed on the generic point, there is no harm in replacing $X$ by $U$ and $X'$ by $\pi^{-1}(U)$, so that the morphism $\pi$ is finite \'etale. 

The induced morphism $\tilde{Y} \times_X X' \to \tilde{Y}$ is \'etale, thus the fibered product $\tilde{Y} \times_X X'$ is normal (\cite[Exp. I, Th\'eor\`eme 9.5 (i)]{SGA1}). The connected components of $\tilde{Y} \times_X X'$ are therefore integral (\cite[\href{https://stacks.math.columbia.edu/tag/033M}{Lemma 033M}]{stacks-project}), and the normalization $\tilde{Y}'$ of $Y'$ is one of them.

The usual dictionary between connected covers and subgroups of the fundamental group implies that the degree of the induced map $\tilde{\pi} \colon \tilde{Y}' \to \tilde{Y}$ is $\# \im (\pi_1^\et(\tilde{Y}) \to G)$. The normalization being a birational map, the degree of $\pi \colon Y' \to Y$ coincides with that of $\tilde{\pi}$, whence the statement.
\end{proof}

\begin{lemma} \label{Lemma:SeparateFiniteSetsWithSubgroups} Let $\Gamma$ be a residually finite group and $F$ a finite subset of $\Gamma$. Then, there is a finite-index normal subgroup $N$ of $\Gamma$ such that the map $F \to \Gamma / N$ is injective.
\end{lemma}

\begin{proof} Consider the finite subset $F':= \{ \gamma \delta^{-1} : \gamma, \delta \in F \} \smallsetminus \{ e\}$ where $e$ is the neutral element of $\Gamma$. 

Saying that the group $\Gamma$ is residually finite amounts to the fact that, for each $\gamma \in \Gamma \smallsetminus \{ e\}$, there is a finite-index normal subgroup $N_\gamma$ to which $\gamma$ does not belong. Then, the finite-index normal subgroup $N := \bigcap_{\gamma \in F'} N_\gamma$ does the job.
\end{proof}

We are now in position to proceed with the proof.

\begin{proof}[Proof of {Theorem \ref{ThmIntro:GeometricPart}}]
To begin with, recall the setup: let $X$ be a normal irreducible projective variety over an algebraically closed field $k$ of characteristic $0$, $L$ an ample line bundle, $U$ an non-empty open subset of $X$ such that the \'etale fundamental group of $U$ is large, and $d \ge 1$ an integer. 

By Lefschetz's principle, the varieties $X$, $U$ and the line bundle $L$ can be defined over a subfield of $k$ finitely generated over $\bbQ$. Since having large \'etale fundamental group is a property that passes to algebraically closed subfields, there is no loss of generality in assuming that $k$ is a subfield of $\bbC$.

Then, according to Proposition \ref{Prop:FinitenessConjugationClasses}, there are finitely many subgroups $H_1, \dots, H_r$ of $\pi_1^\et(U)$ such that, given a normal cycle $f \colon Z \to X$ such that $\deg(Z, f^\ast L) \le d$, the image of the group homomorphism $\pi_1^\et(f^{-1}(U)) \to \pi_1^\et(U)$ is conjugated to some of the subgroups $H_i$. 

Needless to say, by definition of large \'etale fundamental group, each of the subgroups $H_i$ is infinite. By design, the \'etale fundamental group is profinite (in particular, residually finite), thus each finite subset injects into some finite quotient: Lemma \ref{Lemma:SeparateFiniteSetsWithSubgroups} implies that there exists a normal subgroup $N$ of $\pi_1^\et(U)$ such that, for each $i = 1, \dots, r$, the image of $ H_i \to G:=\pi_1^\et(U) / N$ has cardinality $\ge d$. 

Let $U' \to U$ be the finite \'etale cover of $U$ associated with the subgroup $N$. As argued in the proof of Lemma \ref{Lemma:DegreeOfGaloisCoversOnSubvarieties}, the normality of $U$ implies that of $U'$, thus the variety $U'$, \emph{a priori} just connected, is integral. Let $\pi \colon X' \to X$ be the normalization of $X$ in $U'$. By construction, the variety $X'$ is integral normal, and the morphism $\pi$ is Galois of group $G$ and \'etale over $U$.

To see that such a cover fulfills the requirements, let $Y'$ be a positive-dimensional integral closed subvariety of $X'$ meeting $\pi^{-1}(U)$ and set $Y := \pi(Y')$. The projection formula reads
\[ \deg(Y', \pi^\ast L_{\rvert Y'}) = \deg(\pi_{\rvert Y'}) \deg(Y, L_{\rvert Y}),\]
where $\deg(\pi_{\rvert Y'})$ is the degree of the finite map $Y' \to Y$ induced by $\pi$. Of course, if $\deg(Y, L_{\rvert Y}) > d$, then $\deg(Y', \pi^\ast L_{\rvert Y'}) \ge d$. 

Suppose instead $\deg(Y, L_{\rvert Y}) \le d$ and let $\nu \colon \tilde{Y} \to Y$ be the normalization. By the projection formula (or Lemma \ref{Lemma:DegreeAndBirationalMorphisms}),
\[ \deg(\tilde{Y}, \nu^\ast L_{\rvert \tilde{Y}}) = \deg(Y, L_{\rvert Y}) \le d.\]
According to Lemma \ref{Lemma:DegreeOfGaloisCoversOnSubvarieties}, the map $\pi_{\rvert Y'}$ has degree $\# \im (\pi_1^\et(\nu^{-1}(U \cap Y)) \to G)$. On the other hand, by construction, the image of $\pi_1^\et(\nu^{-1}(U \cap Y)) \to \pi_1^\et(U)$ is conjugated to some of the subgroups $H_i$, thus
\[ \# \im (\pi_1^\et(\nu^{-1}(U \cap Y)) \to G) =  \# \im (H_i \to G) \ge d.\]
By ampleness of $L$, the degree $\deg(Y, L_{\rvert Y})$ is a positive integer, thus the projection formula implies $\deg(Y', \pi^\ast L_{\rvert Y'}) \ge d$ as desired.\end{proof}
\section{Arithmetic}

Let $K$ be a number field and $\cO_K$ its ring of integers.

\subsection{Growth rates} \label{sec:GrowthRates}

\subsubsection{Hermitian line bundles} A \emph{Hermitian line bundle} $\bar{\cL}$ on a proper and flat $\cO_K$-scheme $\cX$ is the datum of a line bundle $\cL$ on $\cX$ and, for every embedding $\sigma \colon K \to \bbC$, a continuous metric $\| \cdot \|_\sigma$ on the holomorphic line bundle $\cL_{\rvert \cX_\sigma(\bbC)}$. Moreover, the collection of metrics $\{ \| \cdot \|_\sigma \}$ is supposed to be compatible with complex conjugation. 

For a morphism of $\cO_K$-schemes $f \colon \cY \to \cX$, where the $\cO_K$-scheme $\cY$ is also proper and flat, the pull-back $f^\ast \bar{\cL}$ is defined in the evident way.

\subsubsection{Degree} The \emph{(Arakelov) degree} of a Hermitian line bundle $\bar{\cL}$ on $\Spec \cO_K$ is
\[ \deg \bar{\cL} =  \log \# (\cL / s\cO_K) - \sum_{\sigma \colon K \to \bbC} \log \| s\|_\sigma, \]
where $s$ is a non-zero element of the $\cO_K$-module $\cL$; the quantity above does not depend on its choice because of the product formula.

\subsubsection{Extension of scalars} Let $K'$ be a finite extension  of $K$ and \[\pi \colon \Spec \cO_{K'} \too \Spec \cO_{K}\] the morphism induced by the inclusion of $\cO_K$ in $\cO_{K'}$. Given a Hermitian line bundle $\bar{\cL}$ over $\cO_K$, the Hermitian line bundle $\pi^\ast \bar{\cL}$ over $\cO_{K'}$ has degree 
\begin{equation} \label{Eq:ArakelovDegreeOnAnExtension} \deg \pi^\ast \bar{\cL} = [K' : K] \deg \bar{\cL}. \end{equation}

\subsubsection{Height} Let $\bar{\cL}$ be a Hermitian line bundle on a proper and flat $\cO_K$-scheme $\cX$. By the valuative criterion of properness, a $K$-rational point $P$ of $\cX$ extends uniquely to an $\cO_K$-valued point $\cP$ of $\cX$. The \emph{height} of the point $P$ with respect to the Hermitian line bundle $\bar{\cL}$ is
\[ h_{\bar{\cL}}(P) = \frac{\deg \cP^\ast \bar{\cL}}{[K : \bbQ]} .\]
Replacing the number field $K$ by a finite extension permits to define the height for any point of $\cX$ with values in an algebraic closure $\bar{K}$ of $K$ (equation \eqref{Eq:ArakelovDegreeOnAnExtension} implies that the height is well-defined). A routine variation of the proof of \cite[Theorem B.3.2 (d)]{HindrySilverman} yields:

\begin{proposition} \label{Prop:BoundedDifferenceHeight} Let $\cX$ and $\cX'$ be proper and flat $\cO_K$-schemes endowed respectively with Hermitian line bundles $\bar{\cL}$ and $\bar{\cL}'$. 

Suppose that there exists an isomorphism $f \colon \cX_K \to \cX'_K$ between the generic fibers of $\cX$ and $\cX'$, and an isomorphism $\cL_{\rvert \cX_K} \cong f^\ast \cL'_{\rvert \cX'_K}$ of line bundles over $\cX_K$. Then, 
\[ \sup_{P \in \cX(\bar{K})} |h_{\bar{\cL}}(P) - h_{\cL'}(f(P))| < + \infty.\]
\end{proposition}

A line bundle $\cL$ on a proper and flat $\cO_K$-scheme is \emph{generically ample} if its restriction to the generic fiber $\cX_K$ is ample. In this framework the Northcott property can be stated as follows (see for instance \cite[Theorem B.2.3]{HindrySilverman}):

\begin{proposition} \label{Prop:FinitenessBoundedPoints} Let $\cX$ be a proper and flat $\cO_K$-scheme together with a Hermitian line bundle $\bar{\cL}$ on $\cX$. If $\cL$ is generically ample, then, for a real number $c$, the set 
\[ \{ P \in \cX(K) : h_{\bar{\cL}}(P) \le c \} \]
is finite.
\end{proposition}

\subsubsection{Counting function} \label{sec:CountingFunction} Let $\cX$ be a proper and flat $\cO_K$-scheme, $\bar{\cL}$ a generically ample Hermitian line bundle on $\cX$ and $\cU$ an open subset of $\cX$. For a subring $R$ of $K$ containing $\cO_K$ and a real number $c$, Proposition \ref{Prop:FinitenessBoundedPoints} permits to define 
\[ \nu(\cX, \bar{\cL}, \cU, R, c) := \log^+ \# \{ P \in \cU(R) : h_{\bar{\cL}}(P) \le c \}, \]
where $\log^+ z := \log \max \{1, z \}$ for $z \in \bbR$. The \emph{growth rate} of $R$-points of $\cU$ is
\[ \growth(\cX, \bar{\cL}, \cU, R) := \limsup_{c \to + \infty} \frac{\nu(\cX, \bar{\cL}, \cU, R, c)}{c}. \]
Clearly, such a function is non-decreasing in the variable $R$ (with respect to inclusion).

\subsubsection{Growth rates} Let $X$ be a proper $K$-scheme, $L$ an ample line bundle on $X$ and $U$ an open subset of $X$. Choose a proper and flat $\cO_K$-scheme $\cX$, a Hermitian line bundle $\bar{\cL}$ on $\cX$ and an open subset $\cU$ of $\cX$ whose generic fibers are respectively $X$, $L$ and $U$. By Proposition \ref{Prop:BoundedDifferenceHeight}, the real numbers
\begin{align*}
\grrat_K(X, L, U) &= \growth(\cX, \bar{\cL}, \cU, K), \\
\grint_K(X, L, U) &= \sup_S \growth(\cX, \bar{\cL}, \cU, \cO_{K, S}),
\end{align*}
where the supremum ranges on the finite set of places $S$ of $K$ containing the Archimedean ones, do not depend on the chosen $\cX$, $\bar{\cL}$ and $\cU$. They are called the \emph{growth rate} respectively of \emph{rational} and \emph{integral points} of $U$ (with respect to $X$ and $L$). Clearly,
\[ \grint_K(X, L, U) \le \grrat_K(X, L, U). \]

\begin{remark} Some considerations:
\begin{enumerate}
\item The growth rate of rational and integral points differ in general when $U$ is not proper. For instance, take $X = \bbP^1_K$, $L = \cO(1)$ and $U = \bbP^1_K \smallsetminus \{ 0, 1, \infty \} $. Then,
\begin{align*}
\grrat_K(X, L, U) &= \grrat_K(X, L, X) = 2[K : \bbQ], \\
\grint_K(X, L, U) &= 0,
\end{align*}
because of the finiteness of solutions to the $S$-unit equation (see, for instance, \cite[Chapter 5]{BombieriGubler}).

\item Let $K'$ be a finite extension of $K$, $X'$ and $U'$ the $K'$-schemes deduced respectively from $X$ and $U$ by extending scalars to $K'$, and $L'$ the line bundle on $X'$ deduced from $L$. Then,
\[
\grrat_K(X, L, U) \le \grrat_{K'}(X', L', U') ,
\]
and similarly for the growth rate of integral points. However, as the example above shows, in general the real numbers
$ \grrat_{K'}(X', L', U')$ tend to $\infty$ as soon as the degree of $K'$ does.
\item For an integer $n \ge 1$,
\[ \grrat_K(X, L, U)= n  \grrat_K(X, L^{\otimes n}, U)\]
and similarly for the growth rate of integral points. In particular, these growth rates really depend on the line bundle $L$ and not only on its restriction to $U$.
\end{enumerate}
\end{remark}

\subsection{Proof of Theorem \ref{ThmIntro:RationalGrowthVarietiesSmallDegree}} 

\subsubsection{Statement}
Let $N \ge 1$ be an integer. For a locally closed subvariety $U$ of $\bbP^N_K$ and a real number $c$, with the notation of section \ref{sec:CountingFunction}, let 
\[ \nu_K(U, c) := \nu(\cX, \bar{\cL}, \cU, K, c),\]
where $\cX$ is the Zariski closure of $U$ in $\bbP^N_{\cO_K}$, $\cZ$ is the Zariski closure of $\cX_K \smallsetminus U$ in $\bbP^N_{\cO_K}$, $\cU = \cX \smallsetminus \cZ$, and $\bar{\cL}$ the Hermitian line bundle on $\bbP^N_{\cO_K}$ obtained by endowing the line bundle $\cO_{\bbP^N}(1)$ on $\bbP^N_{\cO_K}$ with metrics, for an embedding $\sigma \colon K \to \bbC$,
\[ \| s\|_\sigma(x) = \frac{|s(x)|_\sigma}{\displaystyle \max_{i = 0, \dots, N} |x_i|_\sigma }, \]
where $s$ is a local section of $\cO(1)$ and $x_0, \dots, x_N$ are the homogeneous coordinates of a point $x$ in $\bbP^N(\bbC)$.

The height function associated with the Hermitian line bundle $\bar{\cL}$ is, for a $K$-rational point $x$ of $\bbP^N_K$,
\[h(x):= h_{\bar{\cL}}(x) = \sum_{v \in \rV_K^0} \log \max_{i = 0, \dots, N} |x_i|_v + \sum_{\sigma \colon K \to \bbC} \log \max_{i = 0, \dots, N} |x_i|_\sigma, \]
where $\rV_K^0$ is the set of finite places of $K$. In the formula above, for a $p$-adic place $v$, the absolute value $|\cdot|_v$ is normalized as $|p|_v= p^{-[K_v : \bbQ_p]}$ where $K_v$ is the completion of $K$ with respect to $v$.

\begin{theorem} \label{Thm:APrioriBound} Let $Z$ be a closed subvariety of $\bbP^N_K$, $\epsilon > 0$ a real number and $n \ge 0$, $D\ge 1$ integers. 

Then, there is a real number $C_{n, D} = C_{n, D}(N, K, Z, \epsilon)$ with the following property: for an integral $n$-dimensional closed subvariety $X$ of $\bbP^N_K$ of degree $\le D$ such that each positive-dimensional integral closed  subvariety in $X$ not contained in $Z$ has degree $\ge d^{\dim Z}$ for some integer $d \ge 1$, and a real number $c \ge [K : \bbQ] \epsilon$, the following inequality holds:
\[ \nu_K(X \smallsetminus Z, c) \le c [K : \bbQ] (1 + \epsilon) \frac{n(n+3)}{2d}  + C_{n, D}.\]
\end{theorem}

Note that, for $n = 0$, the above inequality reads $\nu_K(X \smallsetminus Z, c) \le C_{n, D}$. Before going into the proof of the preceding statement, let us see how it permits to prove Theorem \ref{ThmIntro:RationalGrowthVarietiesSmallDegree}.

\begin{proof}[{Proof of Theorem \ref{ThmIntro:RationalGrowthVarietiesSmallDegree}}] First of all, and rather crucially, notice that the hypotheses and the conclusions are insensitive to taking powers of $L$. Therefore, up to replacing $L$ with a multiple big enough, one may assume that $L$ is very ample. Via the associated embedding $X \to \bbP(\Gamma(X, L)^\vee)$, Theorem \ref{Thm:APrioriBound} can be applied  to give
\[ \grrat_K(X, L, X \smallsetminus Z) \le [K : \bbQ]  \frac{n(n +3)}{2d}, \]
as desired.
\end{proof}

\subsubsection{Proofs} The statement will be deduced by induction from the following:

\begin{theorem}[{\cite[Theorem A]{ChenCrelleII}}] \label{Thm:BrobergChen} Let  $\epsilon > 0$ be a real number and $D \ge 0$ an integer. 

Then, there are positive integers $A_{D} = A_{D}(N, K, \epsilon)$ and $B_{D} = B_D(N, K, \epsilon)$ with the following property: for an integral closed subvariety $X$ of $\bbP^N_K$ of degree $D' \le D$, and a real number $c \ge [K : \bbQ] \epsilon$, the set
\[ \{ x \in X^\reg(K) : h(x) \le c\}\]
can be covered by no more than
\[ A_D \exp\left(  \frac{n + 1}{D'^{1/n}} (1 + \epsilon) [K: \bbQ]c \right)\]
hypersurfaces in $\bbP^N_K$ of degree $\le B_D$ not containing $X$, where $n = \dim X$.
\end{theorem}

\begin{proof}[{Proof of Theorem \ref{Thm:APrioriBound}}]  The proof goes by induction on $n$. For $n = 0$, there is nothing to do, as $X$ is a singleton. 

Suppose $n \ge 1$ and the result true in dimension $< n$. Let $A_D$ and $B_D$ be as in the statement of Theorem \ref{Thm:BrobergChen}, and $R_D$ as in that of Proposition \ref{Prop:DegreeSingularLocus}.

Let $X$ be an integral $n$-dimensional closed subvariety of $\bbP^N_K$ of degree $D' \le D$ such that all positive-dimensional integral closed subvarieties $Y$ of $X$ not contained in $Z$ have degree $\ge d^{\dim Y}$. Quite trivially, such a property is inherited by integral subvarieties $X'$ of $X$ not contained in $Z$: any positive-dimensional integral closed  subvariety $Y$ of $X'$ not contained in $Z$ has degree $\ge d^{\dim Y}$. 

Now, by Theorem \ref{Thm:BrobergChen}, there exist $r$ hypersurfaces $H_1, \dots, H_r$ of $\bbP^N_K$ of degree $\le B_D$ with
\begin{equation} \label{Eq:BoundNumberHypersurfaces} r \le  A_D \exp\left(  \frac{n + 1}{D^{1/n}} (1 + \epsilon) [K: \bbQ]c \right) 
\end{equation}
not containing $X$ such that the set $\{ x \in X^{\reg}(K) : h(x) \le c \}$ is contained in the union of $H_1, \dots, H_r$. For $i= 1, \dots, r$, the hyperplane section $X_i := H_i \cap X$ of $X$ is pure of dimension $n - 1$ and has degree $\le D B_D$. 

Each irreducible component $X'_i$ of $X_i$ has degree $\deg (X'_i, \cO(1)) \le D B_D$ by Lemma \ref{Lemma:DegreeIrreducibleComponents}. Therefore, it is possible to apply the induction hypothesis to such an $X_i'$ and obtain
\[ \nu_K(X_i' \smallsetminus Z, c) \le c [K : \bbQ] (1 + \epsilon)  \frac{(n-1)(n+2)}{2d} + C_{n-1,  D B_D},\]
because $d \le D B_D$.

Since the hyperplane section $X_i$ has at most $R_{DB_D}$ irreducible components by Proposition \ref{Prop:DegreeSingularLocus}, 
\[ \nu_K(X_i \smallsetminus Z, c) \le c [K : \bbQ] (1 + \epsilon)   \frac{(n-1)(n+2)}{2d} + C_{n-1, D B_D} + \log R_{DB_D}.\]
Taking into account \eqref{Eq:BoundNumberHypersurfaces}, and recalling $D' \ge d^n$, the preceding inequality yields
\begin{equation} \label{Eq:BoundNumberRegularPoints} \nu_K(X^{\reg} \smallsetminus Z, c) \le c [K : \bbQ] (1 + \epsilon)   \frac{n(n+3)}{2d} + C'_{n, D},\end{equation}
where $C'_{n, D} := C_{n-1, D B_D} + \log R_{D B_D} + \log A_D$ and the following identity has been noticed:
\[ (n + 1) +  \frac{(n-1)(n+2)}{2} = (n+1) + \sum_{i = 1}^{n-1} (i + 1) = \sum_{i = 1}^{n} (i + 1) =  \frac{n(n+ 3)}{2}.\]

Let $Y_1, \dots, Y_s$ be the irreducible components of the singular locus $X^{\sing}$ of $X$. According to Proposition \ref{Prop:DegreeSingularLocus}, the subvariety $Y_i$ has degree $\le R_D$ and $s \le R_D$. The induction hypothesis, applied to an irreducible component $Y_i$ not contained in $Z$, gives
\[ \nu_K(Y_i \smallsetminus Z, c) \le c [K : \bbQ] (1 + \epsilon)   \frac{(n-1)(n+2)}{2d} +  \tilde{C}_{n-1, R_D},\]
where $\tilde{C}_{n-1, R_D} := \max\{C_{0, R_D}, \dots,  C_{n-1, R_D} \}$, because the irreducible component $Y_i$ has degree $\le R_D$ and dimension $\le n- 1$. In particular,
\begin{equation} \label{Eq:BoundNumberSingularPoints} \nu_K(X^{\sing} \smallsetminus Z, c) \le c [K : \bbQ] (1 + \epsilon)   \frac{(n-1)(n+2)}{2d} + C''_{n, D}, 
\end{equation}
where $C''_{n, D} := \tilde{C}_{n-1, R_D} + \log R_D$.

Combining the inequalities \eqref{Eq:BoundNumberRegularPoints} and \eqref{Eq:BoundNumberSingularPoints} yields the bound in the statement with $C_{n, D} := \log( \exp(C'_{n, D}) + \exp(C''_{n, D}))$.
\end{proof}

\subsection{Growth rates on covers} The last ingredient for the proof of the main theorem is a version of the Chevalley-Weil theorem for growth rates of integral points (Proposition \ref{Prop:BoundGrowthByGaloisCovers}). Its proof is a variation of the classical one (see \cite[4.2]{SerreMordellWeil} or Theorem 2.3 in Corvaja's contribution to \cite{CorvajaCW}).

\subsubsection{Reminder on twists} Let $S$ be a Noetherian scheme and $G$ a finite group. The constant group $S$-scheme with value $G$ is still denoted $G$.

A \emph{principal $G$-bundle} is a faithfully flat finitely presented (\emph{a posteriori} \'etale) $S$-scheme $P$ endowed with an action of $G$ such that the morphism $G \times_S P \to P \times_S P$, $(g, p) \mapsto (gp, p)$ is an isomorphism. The set of isomorphism classes of principal $G$-bundles over $S$ is denoted $\rH^1(S, G)$ or simply $\rH^1(A, G)$ if $S = \Spec A$ is affine.

\begin{proposition} \label{Prop:FinitenessGaloisCohomology} Let $K$ be a number field, $S$ a finite set of places of $K$ containing the Archimedean ones.  Then, 
\[ \# \rH^1(\cO_{K, S}, G) < + \infty.\]
\end{proposition}

\begin{proof} A principal $G$-bundle $P$ over $\cO_{K, S}$ is finite \'etale as an $\cO_{K, S}$-scheme. Thus its generic fiber $P_K$ is the spectrum of a $K$-algebra $A$ whose dimension as a $K$-vector space is $\# G$ and is the product of (finitely many) finite extensions of $K$, all of which are unramified outside $S$. The Hermite-Minkowski bound implies that, for any integer $D \ge 1$, there are only finitely many isomorphism classes of finite extensions of $K$ of degree $\le D$ unramified outside $S$ (\cite[4.1]{SerreMordellWeil}). The statement follows.
\end{proof}

\begin{definition} Let $X$ be an $S$-scheme endowed with an action of $G$ and $P$ a principal $G$-bundle over $S$. The \emph{twist} of $X$ by $P$, if it exists, is the categorical quotient of $X \times_S P$ by the diagonal action $g(x, p) = (gx, gp)$ of $G$.
\end{definition}

Clearly, isomorphic principal $G$-bundles give rise to isomorphic twists. With an abuse of notation, for $t \in \rH^1(S, G)$, let $X_t$ denote the twist of $X$ by a principal $G$-bundle in the isomorphism class $t$. For a scheme $X$ over $P$, the datum of an equivariant action of $G$ is equivalent to a descent datum. Therefore, the theory of faithfully flat (actually, \'etale) descent shows the following existence result:

\begin{proposition}[{\cite[Proposition 4.4.9]{Olsson}}] \label{Prop:ExistenceTwist} Let $X$ be an affine $S$-scheme endowed with an action of $G$ and $P$ a principal $G$-bundle over $S$. Then, the twist of $X$ by $P$ exists.
\end{proposition}

\begin{remark}
The construction of twists is functorial. Namely, let $X$, $Y$ be $S$-schemes endowed with an action of $G$, and $P$  a principal $G$-bundle over $S$ for which the twists of $X$ and $Y$ by $P$ exist. Then, a $G$-equivariant morphism $f \colon X \to Y$ induces a morphism  $f_P \colon X_P \to Y_P$ between twists.
\end{remark}

\subsubsection{Lifting points} A useful construction consists in twisting schemes to lift points. More precisely, let $X$ be an $S$-scheme and $\pi \colon Y \to X$ a principal $G$-bundle over $X$. For an $S$-valued point $x$ of $X$, the scheme-theoretic fiber of $\pi$ at $x$,
\[ P := Y \times_X S,\]
is a principal $G$-bundle over $S$. The twist of $Y$ by $P$ exists as $X$-scheme\footnote{Indeed, it is identified with the twist of $Y$ by the $G$-principal bundle $P \times_S X$, and the latter exists because $\pi$ is a finite morphism (in particular affine) according to the finiteness of $G$.} and the $G$-invariant map $\pi \colon Y \to X$ induces a morphism $\pi_P \colon Y_P \to X$.

\begin{lemma} \label{Lemma:LiftingByTwisting}There is an $S$-valued point $y$ of $Y_P$ such that $\pi_P(y) = x$.
\end{lemma}

\begin{proof} The scheme-theoretic fiber $Y_P \times_X P$ of $\pi_P$ at $x$ is isomorphic to the twist $P_P$ of $P$ by itself. The diagonal embedding $\Delta \colon P \to P \times_S P$ is $G$-equivariant and, taking its quotient by $G$, defines the wanted $S$-valued point $y \colon S \to P_P$.
\end{proof}

\subsubsection{Statement} A finite surjective morphism $f \colon Y \to X$ between algebraic varieties over $K$ is \emph{Galois} if, for each geometric point $\bar{x}$ of $X$, the group $\Aut(f)$  acts transitively on the geometric fiber $Y_{\bar{x}}$.

The group $\Aut(f)$ acts by definition on $Y$. For $t \in \rH^1(K, \Aut(f))$, let $Y_t$ be the twist of $Y$ and $f_t \colon Y_t \to X$ the twist of the $\Aut(f)$-invariant morphism $f$.

\begin{proposition} \label{Prop:BoundGrowthByGaloisCovers} Let $f \colon Y \to X$ be a finite surjective morphism between integral projective varieties over $K$, $U$ an open subset of $X$ over which $f$ is \'etale, and $L$ an ample line bundle on $X$. If the morphism $f$ is Galois, then
\[ \grint_K(X, L, U) \le \max_{t \in \rH^1(K, \Aut(f))} \grint_K(Y_t, f_t^\ast L, f_t^{-1}(U)).\]
\end{proposition}

\begin{proof} Let $G:= \Aut(f)$. Up to taking a power of $L$---an operation that does not affect the statement---the variety  $X$ can be realized as the generic fiber of a projective flat $\cO_K$-scheme $\cX$ over which the line bundle $L$ extends to an ample line bundle $\cL$. Let $\cY$ be the normalization of $\cX$ in $Y$. The induced morphism $\cY \to \cX$, still denoted $f$, is finite surjective and Galois of group $G$. 

 Pick continuous metrics $\{ \| \cdot \|_{L, \sigma} \}_{\sigma \colon K \to \bbC}$ on $\cL$ so that $\bar{\cL} = (\cL, \| \cdot \|_{L, \sigma})$ is a metrized line bundle. Let $Z := X \smallsetminus U$, $\cZ$ its Zariski closure in $\cX$,  $\cU := \cX \smallsetminus \cZ$ and $\cV= f^{-1}(\cU)$.

Let $S$ be a finite set of places of $K$ containing the Archimedean ones. Up to enlarging $S$, one may assume that the morphism $\cV \times_{\cO_K} \cO_{K, S} \to   \cU \times_{\cO_K} \cO_{K, S}$ induced by $f$ is \'etale.  For an isomorphism class of $G$-principal bundles $t \in \rH^1(\cO_{K, S}, G)$, let $\cY_{S, t}$ be the twist of $\cY_S$ by $t$ and $f_{S, t} \colon \cY_{S, t} \to \cX \times_{\cO_K} \Spec \cO_{K, S}$ the induced morphism. Consider the relative normalization $f_t \colon \cY_t \to \cX$ of $\cX$ in $\cY_{S, t}$ (see \cite[\href{https://stacks.math.columbia.edu/tag/035H}{Definition 035H}]{stacks-project}) and set $\cV_t := f_t^{-1}(\cU)$.

\begin{claim} With the notation above,
\begin{equation} \label{Eq:ProofGrowthOfTwists}  
\growth(\cX, \bar{\cL}, \cU, \cO_{K, S}) \le \max_{t \in \rH^1(\cO_{K, S}, G)} \growth(\cY_t, f_t^\ast \bar{\cL}, \cV_t, \cO_{K, S}). 
\end{equation}
\end{claim}

\begin{proof}[Proof of the Claim] According to Proposition \ref{Prop:FinitenessGaloisCohomology}, the set $\rH^1(\cO_{K, S}, G)$ is finite. Therefore, the $\cO_{K}$-scheme 
\[ \tilde{\cY} := \bigsqcup_{t \in \rH^1(\cO_{K, S}, G)} \cY_t\]
is proper. Let $\tilde{f} \colon \tilde{\cY} \to \cX$ the morphism induced by the various twists of $f$. Given an $\cO_{K,S}$-point $x$ of $\cU$, there is an $\cO_{K, S}$-point $\tilde{x}$ of $\tilde{U}$ mapping to $x$. Indeed, the scheme-theoretic fiber \[\cP := \cY \times_\cX \Spec \cO_{K, S}\] of $f$ at $x$ is a principal $G$-bundle over $\cO_{K, S}$. Letting $t$ denote the isomorphism class of $\cP$, Lemma \ref{Lemma:LiftingByTwisting} states that there is an $\cO_{K, S}$-point $y$ of $\cV_t$ such that $f_t(y) = x$. By definition,
\[ h_{f_t^\ast \bar{\cL}}(y) = h_{\bar{\cL}}(x).\]
 In particular,
\[ \growth(\cX, \bar{\cL}, \cU, \cO_{K, S}) \le \growth(\tilde{\cY}, \tilde{f}^\ast \bar{\cL}, \tilde{f}^{-1}(\cU), \cO_{K, S}),\]
whence the claim.
\end{proof}

One concludes by bounding the right-hand side of \eqref{Eq:ProofGrowthOfTwists} by
\[ \max_{t \in \rH^1(K, G)} \grint_{K} (Y_t, f_t^\ast L, V_t),\]
where $V_t = f_t^{-1}(U)$, so that the right-hand side of the so-obtained inequality
\[ \growth(\cX, \bar{\cL}, \cU, \cO_{K, S}) \le \max_{t \in \rH^1(K, G)} \grint_{K} (Y_t, f_t^\ast L, V_t),\]
is independent of $S$.  As $S$ is arbitrary, taking the supremum over all finite sets of places of $K$ (containing the Archimedean ones) finishes the proof.
\end{proof}

\section{Proof of the main theorem}

\subsection{Setup} Let $\bar{K}$ be an algebraic closure of $K$. For a variety $V$ over $K$, let $\bar{V}$ denote the variety over $\bar{K}$ obtained by extending scalars. 

Keep the notation from the statement of the main theorem. There is no loss of generality in assuming $X$ geometrically integral and normal, as $U$ is so. Also, one may assume that $U$ has a $K$-rational point, for the main theorem is trivial otherwise. Let $\epsilon >0$ be a real number and pick an integer $d \ge 1$ such that
\[ \frac{n(n + 3)}{2d} [K:\bbQ] \le \epsilon,\]
where $n = \dim X$. 
\subsection{Geometric input} Thanks to Theorem \ref{ThmIntro:GeometricPart}, there exists a finite surjective morphism $f \colon X' \to \bar{X}$ of algebraic varieties over $\bar{K}$ with $X'$ integral, $f$ \'etale over $\bar{U}$ and, for a positive-dimensional integral closed subvariety $Y'$ of $X'$ not contained in $Z' := f^{-1}(\bar{Z})$,
\[ \deg(Y', f^\ast L_{\rvert Y'}) \ge d^{n}.\] 

\begin{claim} There exists an integral variety $X''$ over $K$ and a Galois finite surjective morphism $g \colon X'' \to X$ of algebraic varieties over $K$  such that $g$ is \'etale over $U$ and the morphism $\bar{g} \colon \bar{X}'' \to \bar{X}$ factors through the morphism $f \colon X' \to \bar{X}$.
\end{claim}

In particular, a positive-dimensional integral  closed subvariety $Y''$ of $X''$ not contained in $g^{-1} (Z)$ satisfies \[\deg(Y'', g^\ast L_{\rvert Y''}) \ge d^n \ge d^{\dim Y''}.\] (Note that $Y''$ is not necessarily geometrically irreducible, but no irreducible component of $\bar{Y}''$ is contained in $g^{-1}(\bar{Z})$.)

\begin{proof}[Proof of the Claim] Argue as in the proof of the implication (iii) $\Rightarrow$ (iv) in \cite[Lemma 3.5.57]{PoonenBook}. First, the map $X' \to \bar{X}$ may be assumed to be Galois, for it suffices to replace the function field $\bar{K}(X')$ of $X'$ by its Galois closure $F$ over the function field $\bar{K}(\bar{X})$ of $\bar{X}$, and $X'$ by its normalization in $F$. 

Let us show that $X'$ comes by extension of scalars from some cover $X'' \to X$ defined over $K$ with $X''$ geometrically integral and normal. For, pick a $K$-rational point of $U$ (which exists by assumption) and a $\bar{K}$-point $x'$ in $f^{-1}(x)$. Let $G$ be the stabilizer of $x'$ in the Galois group $\Gal(\bar{K}(X') / K(X))$. The sought-for $X''$ is obtained as the normalization of $X$ in the finite extension $\bar{K}(X')^G$ of $K(X)$. (See \emph{loc.cit.} for details.)
\end{proof}

\subsection{Arithmetic input} Let $G = \Aut(g)$ be the Galois group of the cover $g$. For an isomorphism class $t$ of principal $G$-bundles over $K$, let $X''_t$ be the twist of the variety $X''$ by $t$ and $g_t \colon X''_t \to X$  that of the morphism $g$. According to Proposition \ref{Prop:BoundGrowthByGaloisCovers},
\[ \grint_K(X, L, U) \le \max_{t \in \rH^1(K, G)} \grint_K(X''_t, g_t^\ast L, g_t^{-1}(U)).\]
For an irreducible subvariety $Y''$ of $X''_t$ not contained in $g_t^{-1}(Z)$,
\[\deg(Y'', g_t^\ast L_{\rvert Y''}) \ge d^{\dim Y''},\]
because, after extending scalars to $\bar{K}$, the cover $g_t$ is isomorphic to $g$. 

Theorem \ref{ThmIntro:RationalGrowthVarietiesSmallDegree} can therefore be applied to the integral projective variety $X''_t$, the closed subvariety $g_t^{-1}(Z)$, and the ample line bundle $g_t^\ast L$, to give
\[ \grrat_K(X''_t, g_t^\ast L, g_t^{-1}(U)) \le \frac{n(n + 3)}{2 d} [K:\bbQ] \le \epsilon.\]
Combining these two inequalities yields $\grint_K(X, L, U) \le \epsilon$, as the growth rate of integral points is lesser than that of rational points. As $\epsilon > 0$ is arbitrary, this concludes the proof. \qed

\small

\bibliography{./biblio}

\bibliographystyle{amsalpha}

\end{document}